\documentclass[12pt, reqno, twoside, letterpaper]{amsart}

\usepackage{BF_stylefile}

\usepackage[usenames,dvipsnames,svgnames,table]{xcolor}
\usepackage[mathscr]{eucal}

\title[Carmichael numbers and the sieve]
      {Carmichael numbers and the sieve}
      
\author[W.\ D.\ Banks]{William D.\ Banks}

\address{Department of Mathematics, 
         University of Missouri, 
         Columbia MO, USA.}

\email{bankswd@missouri.edu}

\author[T.\ Freiberg]{Tristan Freiberg}

\address{Department of Mathematics, 
         University of Missouri, 
         Columbia MO, USA.}

\email{freibergt@missouri.edu}

\date{\today}

\begin{document}


\begin{abstract}
Using the sieve, we show that there are infinitely many Carmichael 
numbers whose prime factors all have the form $p = 1 + a^2 + b^2$ 
with $a,b \in \ZZ$.
\end{abstract}

\maketitle

\begin{center}
{\itshape 
Dedicated to Carl Pomerance on the \\ 
occasion of his 70th birthday
}
\end{center}

\section{Introduction}
 \label{sec:intro}
 
For any prime number $n$, Fermat's little theorem asserts that
\begin{equation}
 \label{eq:FLT} 
 a^n \equiv a \pod{n}
      \qquad (a \in \ZZ).
\end{equation}
Around 1910, Carmichael initiated the study of composite numbers 
$n$ with the property \eqref{eq:FLT}; these are now known as 
{\em Carmichael numbers}.  
The existence of infinitely many Carmichael numbers was first 
established in the celebrated 1994 paper of Alford, Granville and 
Pomerance \cite{AGP}.

Since prime numbers and Carmichael numbers are linked by the 
common property \eqref{eq:FLT}, from a number-theoretic point of 
view it is natural to investigate various arithmetic properties of 
Carmichael numbers.  
For example, Banks and Pomerance \cite{BANPOM} gave a conditional 
proof of their conjecture that there are infinitely many 
Carmichael numbers in an arithmetic progression 
$a + bc$ ($c \in \ZZ$) 
whenever $(a,b) = 1$.  
The conjecture was proved unconditionally by Matom\"aki \cite{MAT} 
in the special case that $a$ is a quadratic residue modulo $b$, 
and using an extension of her methods Wright \cite{WRI} 
established the conjecture in full generality.
The techniques introduced in \cite{AGP} have led to many other 
investigations into the arithmetic properties of Carmichael 
numbers; see 
\cite{%
BANKSetal1,
BANKSetal2,
MCNEW,
BAKERetal,
WRI2,
BAN1,
BANYEA,
GRA,
BANKSetal3,
HAR1,
HAR2,
HALHUN,
HSU,
LOHNIE,
AGP2}
and the references therein.

In this paper, we combine sieve techniques with the method of 
\cite{AGP} to prove the following result.

\begin{theorem}
\label{thm:main}
There exist infinitely many Carmichael numbers whose prime factors 
all have the form $p = 1 + a^2 + b^2$ with some $a,b \in \ZZ$. 
Moreover, there is a positive constant $C$ such that the number of 
such Carmichael numbers not exceeding $x$ is at least $x^C$ 
\textup{(}%
once $x$ is sufficiently large in terms of $C$%
\textup{)}.
\end{theorem}

\begin{remark}
 \label{rem:data}
The Carmichael numbers described in this theorem seem to be quite unusual.
Up to $10^8$, there are only seven such Carmichael numbers, namely
\[
561,
162401,
410041,
488881,
656601,
2433601,
36765901.
\]
By contrast, there are 255 ``ordinary'' Carmichael numbers up to 
$10^8$.

As is well known, whenever $p = 6k + 1$, $q = 12k + 1$ and 
$r = 18k + 1$ are simultaneously prime for some positive integer 
$k$, the number $n = pqr$ is a Carmichael number.
However, no number of this form is a Carmichael number of the type 
described in the theorem, since
$p - 1 = 6k$ and 
$r - 1 = 3\cdot 6k$ cannot both be expressed as a sum of two squares.
\end{remark}

\subsection*{Notation}
 \label{sec:notate}

Aside from notation introduced in situ, let $\PP$ be the set 
of primes, and let $p$ and $q$ always denote primes.

Let   
$
 a \pod{b} \defeq \{a + bc : c \in \ZZ\},  
$
$\ind{S} : \NN \to \{0,1\}$ the indicator function of 
$S \subseteq \NN$, 
$
 \pi(x)
  \defeq 
   {\textstyle \sum_{n \le x} }
    \ind{\PP}(n)  
$
and
$
 \pi(x;b,a)
  \defeq 
  {\textstyle \sum_{n \le x} }
    \ind{\PP \cap \, a \pod{b}}(n).
$
Let 
$\phi,\mu,\omega,P^+ : \NN \to \NN$ be the Euler, M\"obius, number 
of distinct prime divisors and greatest prime divisor functions 
($\omega(1) \defeq 0$ and $\gp{1} \defeq 1$).
Let $\log_n : [1,\infty) \to [1,\infty)$ be the $n$th iterated logarithm, i.e.,
$
 \log_1 x \defeq \max\{1,\log x\}
$
and 
$
 \log_{n+1} x \defeq \log_1(\log_n x)
$.

Let expressions of the form
$f(x) = O(g(x))$,  
$f(x) \ll g(x)$ and 
$g(x) \gg f(x)$  
signify that $|f(x)| \le c|g(x)|$ for all sufficiently large $x$, 
where $c > 0$ is an absolute  constant.
The notation $f(x) \asymp g(x)$ indicates that $f(x) \ll g(x) \ll f(x)$.
We also let $f(x) = O_A(g(x))$ etc.\ 
have the same meanings with $c$ depending on a parameter $A$. %
Finally, let $o_{x \to \infty}(1)$ (or simply $o(1)$ if $x$ is clear in context) 
denote a quantity that tends to zero as
$x$ tends to infinity.

\section{AGP setup}
 \label{sec:AGP}

Let $\BB \defeq\{1, 5, 13, 17, 25, \ldots \}$ be the 
multiplicative semigroup of the natural numbers generated by the set of primes
$\PP \cap 1 \pod{4}$, and let  
\[
 \pi(x,y)
  \defeq 
   \# 
    \{p \in \BB \cap [2,x] : \gp{p-1} \le y\}.
\]
\begin{definition}
 \label{def:E}
Let $\mathcal{E}$ be the set of numbers $E$ in $(0,1)$ for which 
there exist $x_1(E),\gamma_1(E) > 0$ such that for all 
$x \ge x_1(E)$, the inequality
\begin{equation}
 \label{eq:piE}
  \pi(x,x^{1 - E})
   \ge 
    \gamma_1(E)\pi(x)
\end{equation}
holds.
\end{definition}
\begin{definition}
 \label{def:ell}
Given $T \ge 3$, let $\ell(T)$ be the integer given in terms of 
putative Siegel zeros%
\footnote{%
We take license with the term ``Siegel zero'' --- cf.\ 
Lemma \ref{lem:4.1} below for a precise statement.
}
in Lemma \ref{lem:4.1} below.
\end{definition}

\begin{definition}
 \label{def:B}
For any fixed positive constants $A,A'$, let
$\mathcal{B}=\mathcal{B}(A,A')$ denote
the set of numbers $B\in(0,1)$ for which the following holds.
There exists $x_2(B)$ such that for all $x \ge x_2(B)$
we have  
\begin{equation}
 \label{eq:0.3AGP}
  \frac{A^{-1}dx^{1 - B}y^{-1}}{\phi(d)\log(dx^{1-B})}
   \le 
    \sqrt{\log x}
     \sum_{\kappa \le x^{1 - B}y^{-1}}
      \ind{\BB}(\kappa)
       \ind{\PP}(2d\kappa + 1)  
        \le
         \frac{A'dx^{1 - B}y^{-1}}{\phi(d)\log(dx^{1-B})}
\end{equation}
whenever 
$d \in \BB \cap [1,x^By]$, 
$|\mu(d)| = 1$, 
$\gp{d},y \le x^{B/\log_2 x}$ and 
$(d,\ell(x^B)) = 1$.
\end{definition}

%
Matom\"aki \cite[Lemma 2]{MAT} has shown that 
$
  \mathcal{E} \supseteq (0,1/2).
$
By Lemma \ref{lem:sievelem} below, if $A,A'$ are 
sufficiently large%
\footnote{%
Although we do not give details, one can show that $A = 50$ and $A' = 1$ 
suffice.
We do not compute a value for $\beta$.
}
and $\beta$ is sufficiently small (depending on $A,A'$), then 
$
  \mathcal{B} \supseteq (0,\beta).
$
Consequently, the following analogue of \cite[Theorem 4.1]{AGP}\break immediately implies Theorem \ref{thm:main}.
Its proof relies on Lemma \ref{thm:3.1AGP} below, which is 
itself analogous to \cite[Theorem 3.1]{AGP}. 

\begin{theorem}
 \label{thm:4.1AGP}
Let $C(x)$ denote the number of Carmichael numbers up to $x$ all 
of whose prime divisors $p$ are such that $(p - 1)/2 \in \BB$.
For each $E \in \mathcal{E} \cap (4/9,1)$, $B \in \mathcal{B}$ and 
$\epsilon > 0$, there is a number $x_4(E,B,\epsilon)$, such that 
whenever $x \ge x_4(E,B,\epsilon)$, we have 
$C(x) \ge x^{EB - \epsilon}$.
\end{theorem}

\begin{lemma}
 \label{thm:3.1AGP}
Fix any $B \in \mathcal{B}$.
There exists $x_3(B)$ such that the following holds for  all
$x \ge x_3(B)$ and any integer $L$ satisfying hypotheses 
\textup{(H1) --- (H5)} below.
There is some $k \in [1,x^{1 - B}] \cap \BB$ with 
$(k,L) = 1$ such that 
\begin{equation*}
 4A
  (\log x)^{3/2}
   \sum_{d \mid L, \, 2dk + 1 \le x}
    \ind{\PP}(2dk + 1)
     >
      \#\Br{d \mid L : d \le x^B}.
\end{equation*}
%
Our hypotheses are the following:
\begin{itemize}

\item[\textup{(H1)}]
 $L \in \BB$ and $|\mu(L)| = 1$;
\item[\textup{(H2)}]
 $\gp{L} \le x^{B/\log_2 x}$;
\item[\textup{(H3)}]
 $(L,\ell(x^B)) = 1$;
\item[\textup{(H4)}]
for any $d \mid L$ with $d \le x^B$, the bound
$
 \textstyle 
 16A\sqrt{\log x} \sum_{q \mid d} 1/q \le 1 - B
$ holds;
\item[\textup{(H5)}] we have
$
 \textstyle 
 24AA' \sum_{q \mid L} 1/q \le 5(1 - B).
$
\end{itemize}
\end{lemma}

\begin{proof}
Let $x \ge x_3(B)$ with $x_3(B)$ sufficiently large (to be 
specified).
We have 
\[
  \sums[\kappa \le x^{1-B}][(\kappa,L) = 1]
   \ind{\BB}(\kappa)
    \sum_{d \mid L, \, d \le x^B} 
     \ind{\PP}(2d\kappa + 1)
 =
     \sum_{d \mid L, \, d \le x^B}
      \hspace{3pt}
       \sums[\kappa \le x^{1 - B}][(\kappa,L) = 1]
        \ind{\BB}(\kappa)
         \ind{\PP}(2d\kappa + 1), 
\]
so there must be some $k \in [1,x^{1 - B}] \cap \BB$ with 
$(k,L) = 1$ for which 
\begin{equation}
 \label{eq:thm3.1AGPi}
 x^{1 - B}
  \sum_{d \mid L, \, d \le x^B}
   \ind{\PP}(2dk + 1) 
    \ge 
     \sum_{d \mid L, \, d \le x^B}
      \hspace{3pt}
       \sums[\kappa \le x^{1 - B}][(\kappa,L) = 1]
        \ind{\BB}(\kappa)
         \ind{\PP}(2d\kappa + 1).     
\end{equation}
Let $d \mid L$, $d \le x^B$.
Note that  
$d$ is squarefree, 
$\gp{d} \le x^{B/\log_2 x}$ and  
$(d,\ell(x^B)) = 1$.
Observe that
\begin{align}
 \label{eq:thm3.1AGPii}
 \begin{split}
 &
 \sums[\mathclap{\kappa \le x^{1 - B}}]
      [\mathclap{(\kappa,L) = 1}]
  \,
  \ind{\BB}(\kappa)
   \ind{\PP}(2d\kappa + 1)
 \\
 & \hspace{45pt}
 \ge
  \sum_{\mathclap{\kappa \le x^{1 - B}}}
   \ind{\BB}(\kappa)
    \ind{\PP}(2d\kappa + 1)
  -
    \sum_{q \mid L}
     \sum_{mq \le x^{1 - B}}
      \ind{\BB}(mq)
       \ind{\PP}(2dmq + 1).
  \end{split}
\end{align}
We can assume that $x_3(B) \ge x_2(B)$;
hence by \eqref{eq:0.3AGP} we have  
\begin{equation}
 \label{eq:3.2AGP}
 A 
  \sqrt{\log x}
   \sum_{\kappa \le x^{1 - B}}
    \ind{\BB}(\kappa)
     \ind{\PP}(2d\kappa + 1)
      \ge  
       \frac{dx^{1 - B}}
            {\phi(d)\log(dx^{1-B})}
         \ge 
          \frac{dx^{1-B}}{\phi(d)\log x}.
\end{equation}
Now fix $q \mid L$ for the moment,
and consider the sum on $mq \le x^{1 - B}$ in 
\eqref{eq:thm3.1AGPii}.
Note that 
\[
 \sum_{mq \le x^{1 - B}}
  \ind{\BB}(mq)
   \ind{\PP}(2dq m + 1)
    \le
     \pi(2dx^{1 - B} + 1;dq,1)
      \le 
       \pi(2dx^{1 - B};dq,1) + 1.
\]
The Brun--Titchmarsh inequality of Montgomery and Vaughan 
\cite{MV} gives 
\[
  \pi(dx^{1 - B};2dq;1) 
   <
    \frac{4dx^{1 - B}}{\phi(dq)\log(x^{1 - B}/q)} 
     \le
      \frac{8}{q(1 - B)}
       \frac{dx^{1 - B}}{\phi(d)\log x} 
        - 1,
\]
provided $x_3(B)$ is sufficiently large, which we assume (recall that
$q \le x^{B/\log_2 x}$).
Using (H4) it follows that
\begin{equation}
 \label{eq:thm3.1AGPiii}
 2A
 \sqrt{\log x}
  \sum_{q \mid d}
   \sum_{mq \le x^{1 - B}}  
    \ind{\BB}(mq)
     \ind{\PP}(2dmq + 1)
      < 
       \frac{dx^{1 - B}}{\phi(d)\log x}.
\end{equation}
Now suppose $q \nmid d$.
For such $q$ we have $dq \mid L$, $|\mu(dq)| = 1$, $\gp{dq} \le x^{B/\log_2 x}$;
therefore, applying
\eqref{eq:0.3AGP} ($d \mapsto dq$, $y \mapsto q$) and 
noting that $q/\phi(q) \le 6/5$ for all $q \ge 5$, it follows that
\[
  \sqrt{\log x} \hspace{3pt}
   \sum_{\mathclap{m \le x^{1 - B}/q}}  
    \ind{\BB}(m)
     \ind{\PP}(2(dq)m + 1)
      \le 
       \frac{A'dx^{1-B}}{\phi(dq)\log(dqx^{1-B})}
        \le 
         \frac{6A'}{5q(1 - B)}
          \frac{dx^{1-B}}{\phi(d)\log x}.
\]
Since $\ind{\BB}(mq) = \ind{\BB}(m)$ we deduce from (H5) that
\begin{equation}
 \label{eq:thm3.1AGPiv}
  4A
  \sqrt{\log x}  
   \sum_{q \mid L,~q \nmid d}~
    \sum_{mq \le x^{1 - B}}
     \ind{\BB}(mq)\ind{\PP}(2dmq + 1)
      \le 
       \frac{dx^{1-B}}{\phi(d)\log x}.
\end{equation}
Combining \eqref{eq:thm3.1AGPii} -- \eqref{eq:thm3.1AGPiv} 
we see that
\[
 4A\sqrt{\log x} 
   \sums[\kappa \le x^{1 - B}]
        [(\kappa,L) = 1]
  \,
   \ind{\BB}(\kappa)
    \ind{\PP}(2d\kappa + 1)
     > 
      \frac{dx^{1-B}}{\phi(d)\log x}      
       \big(4 - 2 - 1\big)    
        \ge 
         \frac{x^{1 - B}}{\log x},
\]
and combining this with \eqref{eq:thm3.1AGPi} we obtain the stated result.
\end{proof}

\begin{proof}[Proof of Theorem \ref{thm:4.1AGP}]
Minor modifications notwithstanding, the proof follows that of 
\cite[Theorem 4.1]{AGP} verbatim, so let us only set up the proof 
here.
Let $E \in \mathcal{E}$, $B \in \mathcal{B}$, $\epsilon > 0$.
We can assume that $\epsilon < EB$.
Let $\theta \defeq (1 - E)^{-1}$ and let $y \ge 2$ be a parameter.
Put
\[
 \mathcal{Q}
  \defeq 
   \{ 
     q
      \in \BB \cap (y^{\theta}/\log y, y^{\theta}] :
       \gp{q-1} \le y
   \},
\]
and let $\ell$ be a positive integer (to be specified) satisfying 
$\log \ell \ll y^{\theta}/\log y$.
By \eqref{eq:piE} we have
\[
  |\mathcal{Q}\setminus\{q : q \mid \ell\}|
   \ge
    \frac{1}{2}
     \gamma_1(E)
      \frac{y^{\theta}}{\log(y^{\theta})}
\]      
for all large $y$ (we have 
$\pi(y^{\theta}/\log y) \ll y^{\theta}/(\log(y^{\theta})\log y)$ 
using Chebyshev's bound, as well as the well-known bound 
$\omega(\ell) \ll (\log \ell)/(\log_2 \ell)$).
Let 
$
 L
  \defeq  
   \prod_{q \in \mathcal{Q}, \, q \nmid \ell}
    q
$; then 
\[
  \log L
   \le 
    |\mathcal{Q}|
      \log(y^{\theta})
       \le 
        \pi(y^{\theta})\log(y^{\theta})
         \le 
           2y^{\theta}
\]           
for all large $y$. 
%
%
%
Let $\delta \defeq \epsilon \theta/(4B)$ and let 
$x \defeq \e^{y^{1 + \delta}}$.
We have 
\[
 \sum_{q \mid L}
  \frac{1}{q}
   \le 
    \sums[{q \in (y^{\theta}/\log y, y^{\theta}]}]
     \frac{1}{q}
      \le 
       2
       \frac{\log_2 y}{\theta\log y}
        \le
         \frac{5(1 - B)}{24AA'}
\]
for all sufficiently large $y$.
For any $d \mid L$ with $d \le x^B$ we have 
$\omega(d) \le 2\log x/\log_2 x$ (if $x$ is large enough), and therefore
\[
 \sum_{q \mid d}
  \frac{1}{q}
   \le 
    \frac{\log y}{y^{\theta}}
     \frac{2\log x}{\log_2 x} 
      <
       \frac{2\log x}{(\log x)^{\theta/(1 + \delta)}}
        < 
         \frac{1 - B}{16A\sqrt{\log x}}
\]
for all large $y$ {\em provided that} $\theta/(1 + \delta) > 3/2$. 
Since
$$
4\delta = \epsilon\theta/B < \theta E = E/(1 - E),$$
we have 
\[
 2\theta - 3\delta 
  =
   2\bigg(1 + \frac{E}{1 - E}\bigg) - 3\delta
    > 
     2\bigg(1 + \frac{E}{1 - E}\bigg) 
      - \frac{3E}{4(1 - E)}
       =
        2 + \frac{5E}{4(1 - E)},
\]
and this is greater than three (and hence $\theta/(1 + \delta) > 3/2$ as required) 
whenever $5E/\big(4(1 - E)\big) > 1$, i.e., $E > 4/9$, which we assume. 

We now specify that $\ell \defeq \ell(x^B)$.
We clearly have $\ell(x^B) \le x^B$ (cf.\ Lemma~\ref{lem:4.1}), 
so the requirement that $\log \ell \ll y^{\theta}/\log y$ is 
satisfied:
\[
 \log \ell 
  \le \log x 
   = y^{1 + \delta} 
    < y^{2\theta/3} 
     \ll y^{\theta}/\log y.
\]

We can apply Lemma \ref{thm:3.1AGP} with $B,x,L,\ell=\ell(x^B)$.
Thus, for all sufficiently large values of $y$, there is an integer 
$k \in \BB$ coprime to $L$, for which the set $\mathcal{P}$ of 
primes $p \le x$ with $p = 2dk + 1$ for some divisor $d$ of $L$, 
satisfies 
\[
  |\mathcal{P}|
   \ge 
    \frac{\#\{d \mid L : d \le x^B\}}{4A(\log x)^{3/2}}.
\]
We leave the reader to pursue 
the remainder of the proof in \cite{AGP}.
\end{proof}

\section{The sieve}
 \label{sec:sieve}

\subsection*{Notational caveat} 
This section can be read independently of \S \ref{sec:AGP}, and 
below $A,B,d,k$ are not the same as in \S \ref{sec:AGP}.

\subsection*{Level of distribution}
 \label{subsec:bomvin}
We first quote part of \cite[Lemma 4.1]{BFM14}, which gives a 
qualitative extension of the classical (exceptional) zero-free 
region for Dirichlet $L$-functions in the case of smooth moduli.
Its proof uses bounds for character sums to smooth moduli due to  Chang \cite{CHA14}.

\begin{lemma}
 \label{lem:4.1}
Let $T \ge 3$.
Among all primitive Dirichlet characters $\chi \bmod \ell$ to 
moduli $\ell$ satisfying $\ell \le T$ and 
$\gp{\ell} \le T^{1/\log_2 T}$, there is at most one for which the 
associated $L$-function $L(s,\chi)$ has a zero in the region 
\begin{align}
 \label{eq:4.6} 
  \Re(s) > 1 - c \log_2 T /\log T, \quad 
  |\Im(s)| \le \exp\big(\sqrt{\log T}/\log_2 T\big),
\end{align}
where $c > 0$ is a certain \textup{(}small\textup{)} absolute 
constant.
If such a character $\chi \bmod \ell$ exists, then $\chi$ is real 
and $L(s,\chi)$ has just one zero in the region \eqref{eq:4.6}, 
which is real and simple, and we set $\ell(T) \defeq \ell$. 
Otherwise we set $\ell(T) \defeq 1$.
\end{lemma}

\begin{remark}
 \label{rem:excmod}
If $\chi \bmod \ell$ is real and primitive, then 
$\ell = 2^{\nu}\lh$ where $\nu \le 3$ and $\lh$ is odd and 
squarefree.
By Siegel's theorem \cite[\S21, (4)]{DAV}, if $\beta$ is any real 
zero of $L(s,\chi)$ then $\ell \gg_A (1 - \beta)^{-A}$ for any 
$A > 1$.
Hence, if $\ell = \ell(T)$ is as in
Lemma~\ref{lem:4.1} and  $\ell \ne 1$, then 
\begin{equation}
 \label{eq:omegaqT}
  \ell
   \gg_A
   (\log \ell/\log_2 \ell)^A.  
\end{equation}
The implicit constant is ineffective for $A > 2$, but it is effective 
for $2 \ge A > 1$, and consequently the implicit constant in 
\eqref{eq:4.10} below is effective for $A < 2$. 
\end{remark}

The following statement is a consequence of \cite[Theorem 4.1]{BFM14}, whose 
proof combines standard zero density estimates with the zero free 
region for smooth moduli given in Lemma~\ref{lem:4.1}.

\begin{theorem} 
 \label{thm:4.2}
Fix $\eta > 0$.
Let $x \ge 3^{1/\eta}$ be a number, and let $k \ge 1$ be a 
squarefree integer such that  
$
 P^{+}(k) < x^{\eta/\log_2 x}
$, 
$
 k < x^{\eta} 
$
and $(k,\ell) = 1$, where $\ell \defeq \ell(x^{\eta})$ as in Lemma 
\ref{lem:4.1}.
If $\eta = \eta(A,\delta)$ is sufficiently small in 
terms of any fixed $A > 0$ and $\delta \in (0,1/2)$, then 
\begin{align}
 \label{eq:4.10}
  \sum_{r \le \sqrt{x}/x^{\delta}}
   \max_{(a,kr) = 1}
    \left| \pi(x;kr,a) -  \frac{\pi(x)}{\phi(kr)} \right|
     \ll_{\delta,A}
      \frac{x}{\phi(k)(\log x)^A}.
\end{align}
\end{theorem}

\begin{proof}
Let us write $\Delta(x;kr,a)$ for $\pi(x;kr,a) - \pi(x)/\phi(kr)$.
The bound 
\begin{equation}
 \label{eq:BFMBVT}
  \sums[r \le \sqrt{x}/x^{\delta}][(r,\gp{\ell}) = 1]
   \max_{(a,kr) = 1}
   |\Delta(x;kr,a)|
     \ll_{\delta,A}
      \frac{x}{\phi(k)(\log x)^A}
\end{equation}
is%
\footnote{%
Actually, in \cite[Theorem 4.1]{BFM14} the primes are counted with 
a logarithmic weight, from which one can deduce, via partial 
summation, the bound as stated in \eqref{eq:BFMBVT},
and this is the form in which the bound is ultimately used in \cite{BFM14}.
}
\cite[Theorem 4.1]{BFM14} in our notation, except that we have the 
stronger hypothesis that $(k,\ell) = 1$,
whereas in \cite{BFM14} it is only
assumed that $(k,\gp{\ell}) = 1$. 
If $\ell = 1$ then we are done, so let us assume $\ell \ne 1$.
By Remark \ref{rem:excmod}, $\ell = 2^{\nu}\lh$, where 
$\nu \le 3$ and $\lh$ is a product of 
$O(\log x^{\eta}/\log_2 x^{\eta})$ distinct odd primes.
The bound \eqref{eq:BFMBVT} holds if $\gp{\ell}$ is replaced by 
any prime divisor of $\ell$, as is manifest from the proof of 
\cite[Theorem 4.1]{BFM14} (the crux being that $\ell \nmid r$). 
Summing over the prime divisors of $\lh$, replacing $A$ by 
$A + 1$ in \eqref{eq:BFMBVT}, and recalling that $\eta$ depends only 
on $A$ and $\delta$, we deduce that 
\begin{equation}
 \label{eq:BFMBVTii}
  \sums[r \le \sqrt{x}/x^{\delta}][\lh\,\nmid\,r]
   \max_{(a,kr) = 1}
   |\Delta(x;kr,a)|
     \ll_{\delta,A}
      \frac{x}{\phi(k)(\log x)^A}.
\end{equation}

On the other hand, using $\pi(x) \ll x/\log x$ together with the Brun--Titchmarsh 
inequality \cite[(13.3) et seq.]{FI10} we obtain that, uniformly 
for $r \le \sqrt{x}$ with $\lh \mid r$ and $(a,kr) = 1$,  
\[
 \Delta(x;kr,a)
  \ll
   \frac{x}{\phi(kr)\log x}.
\]
For any such $r$, write $r=\lh r_1r_2$, where
$r_1$ is composed of primes dividing $\ell$,
and $(r_2,\ell)=1$.  Note that $r_1 \le \sqrt{x}/(r_2\lh )$,
$(kr_2,\lh)=1$ (since $(k,\ell)=1$),
and $\phi(kr)\ge\phi(k)\phi(\lh)\phi(r_1)$;
therefore,
\begin{equation}
 \label{eq:qdivr}
  \sums[r \le \sqrt{x}/x^{\delta}][\lh\,\mid\,r]
   \max_{(a,kr) = 1}
   |\Delta(x;kr,a)|
     \ll
      \frac{x}{\phi(k)\phi(\lh\hskip1pt)\log x}
       \sum_{r_1 \le \sqrt{x}}
        \frac{1}{\phi(r_1)}
    \ll
     \frac{x}{\phi(k)\phi(\lh\hskip1pt)}.
\end{equation}
Since $\ell/\phi(\ell) \ll \log_2 \ell \ll \log_2 x^{\eta}$ and 
$\ell \gg_A (\log x^{\eta}/\log_2 x^{\eta})^A$ by 
\eqref{eq:omegaqT}, we see that 
\[
 1/\phi(\lh)
  \ll
   (\log_2 x^{\eta})^{A + 1}/(\log x^{\eta})^A,
\]
thus combining \eqref{eq:BFMBVTii} with \eqref{eq:qdivr} gives the 
result (with $A$ replaced by any smaller constant).
\end{proof}

\subsection*{Setup \& key estimate}
 \label{subsec:sievesetup}
Equipped with our level of distribution result, establishing our 
key estimate involves a routine application of the semi-linear 
sieve and a ``switching trick'' (as in \cite[Theorem 14.8]{FI10}).
We are to sieve a sequence of primes in arithmetic progression by 
the primes in $\PP \cap 3 \pod{4}$.

For $x \ge 3$, let  
\[
 P(x) \defeq \prods[p < x][p \equiv 3 \pod{4}] p, 
\]
let
\begin{equation}
 \label{eq:def:VW}
  V(x)
   \defeq 
    \prods[p < x][p \equiv 3 \pod{4}]
     \bigg(1 - \frac{1}{\phi(p)}\bigg) 
   =
    \prods[p < x][p \equiv 3 \pod{4}]
     \bigg(1 - \frac{1}{p}\bigg) 
       \bigg(1 + \frac{1}{p(p-2)}\bigg)^{-1}
\end{equation}
and let
\[
  W(x)
   \defeq 
    \prods[p < x][\mathclap{p \equiv 1 \pod{4}}]
     \bigg(
      1 + \frac{1}{\phi(p)} + \frac{1}{\phi(p^2)} + \cdots 
     \bigg)
  =
    \prods[p < x][p \equiv 1 \pod{4}]
     \bigg(1 - \frac{1}{p}\bigg)^{-1}
      \hspace{-3pt}
       \bigg(1 + \frac{1}{p(p-1)}\bigg).
\]
For future reference, we record here that by Mertens' theorem 
one has
\begin{equation}
 \label{eq:WonV}
  W(x)/V(x)
   = {\textstyle \frac{1}{2}}
    A_1A_3\e^{\gamma}\log x + O(1), 
\end{equation}
where 
\begin{equation}
 \label{eq:def:A0}
  A_1
   \defeq 
    \prod_{p \equiv 1 \pod{4}}
     \bigg(1 + \frac{1}{p(p-1)}\bigg)
 \quad 
  \text{and}
   \quad 
    A_3
     \defeq 
      \prod_{p \equiv 3 \pod{4}}
       \bigg(1 + \frac{1}{p(p-2)}\bigg).
\end{equation}
By Mertens' theorem we also have, for $2 \le x < y$ and $j = 1,3$, 
\begin{equation}
 \label{eq:mert1p}
 \sums[x \le p < y][p \equiv j \pod{4}]
  \frac{1}{p}
   =
  \frac{1}{2}\log\bigg(\frac{\log y}{\log x}\bigg)
   +
    O\bigg(\frac{1}{\log x}\bigg)
 \le
  \frac{\log(y/x)}{2\log x}
   \bigg(1 + O\bigg(\frac{1}{\log(y/x)}\bigg)\bigg),
\end{equation}
and furthermore,
\begin{equation}
 \label{eq:sievedim}
 \frac{V(x)}{V(y)}, 
  \frac{W(y)}{W(x)}
   =
    \bigg(\frac{\log x}{\log y}\bigg)^{1/2}
     \bigg(1 + O\bigg(\frac{1}{\log y}\bigg) \bigg). 
\end{equation}
Indeed, we actually have (cf.\ \cite[(14.21)--(14.24)]{FI10})
\begin{equation}
 \label{eq:mert34}
  1/V(x) 
   =
    2A_3B\sqrt{(\e^{\gamma}/\pi)\log x}
     \big(1 + O(1/\log x)\big)    
\end{equation}
and 
\[
  W(x)
   =
    (\pi A_1/4B)
     \sqrt{(\e^{\gamma}/\pi) \log x}
      \big(1 + O(1/\log x)\big), 
\]      
where 
\[
 B 
  \defeq 
   \frac{1}{\sqrt{2}}
    \prod_{p \equiv 3 \pod{4}}
     \bigg(1 - \frac{1}{p^2}\bigg)^{-1/2}
   =
   \frac{\pi}{4}
    \prod_{p \equiv 1 \pod{4}}
     \bigg(1 - \frac{1}{p^2}\bigg)^{1/2}     
      =
       0.764223\ldots
\]
is the Landau--Ramanujan constant.
Finally, let $f(s)$ and $F(s)$ be the continuous solutions to the 
following system of differential-difference equations:
\[
\setlength{\tabcolsep}{5pt}
\begin{tabular}%
{%
  >{$}r<{$} 
  @{ $=$ } 
  >{$}l<{$}
  @{ \hspace{15pt} }
  >{$}l<{$} 
  @{ \hspace{18pt} }    
  >{$}r<{$} 
  @{ $=$ }  
  >{$}l<{$} 
  @{ \hspace{15pt} }  
  >{$}l<{$}
}
 \sqrt{s}F(s) & 2\sqrt{\e^{\gamma}/\pi} & (0 \le s \le 2), & (\sqrt{s}F(s))' & f(s - 1)/2\sqrt{s} & (s > 0) \\
         f(1) & 0                       &                  & (\sqrt{s}f(s))' & F(s - 1)/2\sqrt{s} & (s > 1). 
\end{tabular}
\]
For $1 \le s \le 3$ we have \cite[p.275]{FI10} 
\begin{equation}
 \label{eq:fs}
 \frac{\sqrt{s}f(s)}{\sqrt{\e^{\gamma}/\pi}}
   =
     \int_1^s 
      \frac{\dd t}{\sqrt{t(\log t)}}
   =
     \log\br{1 + 2(s - 1) + 2\sqrt{s(s-1)}}.
\end{equation}

\begin{lemma}
 \label{lem:sievelem}
Fix $\eta > 0$.
Let $x \ge 3^{1/\eta}$ be a number, and let $k \ge 1$ be a  
squarefree integer, such that
$
 P^{+}(k) < x^{\eta/\log_2 x}
$, 
$
 k < x^{\eta} 
$
and $(k,\ell) = 1$, with $\ell \defeq \ell(x^{\eta})$ as in Lemma 
\ref{lem:4.1}.
If $k \in \BB$ and $\eta$ is sufficiently small, then   
\begin{align}
 \label{eq:sieveup}
  \sum_{m \le x}
   \ind{\BB}(m)
    \ind{\PP}(2km + 1)
     \asymp 
      \frac{kx}{\phi(k)(\log x)^{3/2}}.
\end{align}
\end{lemma}

\begin{proof} 
Let $k\in\BB$ be fixed.
Note that $(k,2P(x)) = 1$. 
As $\ind{\BB}(m) = 1$ implies that $m \equiv 1 \pod{4}$, and thus
$2km + 1 \equiv 3 \pod{8}$, we can assume that 
our sum is over $m$ for which $2km + 1 \equiv j \pod{8k}$ for 
some reduced residue $j \pod{8k}$, with $j \equiv 3 \pod{8}$ and 
$j \equiv 1 \pod{k}$. 
Thus, we want to sift the sequence
$\mathscr{A} \defeq (\ind{\PP \cap \, j \pod{8k}}(2km+1))$ by the 
primes in $\PP \cap 3 \pod{4}$, and the sum in 
\eqref{eq:sieveup} is equal to $S(\mathscr{A},\sqrt{x})$, where 
\[
 S(\mathscr{A},z)
  \defeq 
   \sums[m \le x][(m,P(z)) = 1] 
    \ind{\PP \cap \, j \pod{8k}}(2km + 1) 
\]
is our sifting function. 

Let $z < x$.
Suppose $d \mid P(z)$ and note that $(d,2k) = 1$ (since 
$2 \nmid P(z)$ and $(k,P(z)) = 1$).
Thus, $d \mid m$ if and only if $2km + 1 \equiv 1 \pod{d}$, and so
\[
 \mathscr{A}_d(x)
  \defeq 
   \sums[m \le x][d \mid m] 
    \ind{\PP \cap \, j \pod{8k}}(2km + 1)
    =
    \pi(2kx + 1;8dk,h)
    =
     g(d)X
      + 
       r_d  
\]
for some reduced residue $h \pod{8dk}$ with $h \equiv j \pod{8k}$ 
and $h \equiv 1 \pod{d}$, and where 
$
 X 
  \defeq 
   \pi(2kx)/\phi(8k)
$,
$ 
  g(d) \defeq 1/\phi(d) 
$ 
and 
\[
 r_d 
  \defeq 
   \mathscr{A}_d(x)
    -
     g(d)X
  =
   \pi(2kx + 1;8dk,h) - \pi(2kx)/\phi(8dk).
\]
Now set $\delta \defeq 1/3890$.
(The argument below works for any sufficiently small $\delta$.)
By Theorem \ref{thm:4.2} ($x \mapsto 2kx$) our sequence 
$\mathscr{A}$ has level of distribution 
$D \defeq \sqrt{x}/x^{\delta}$, and we have 
\[
 R(D,z)
  \defeq 
   \sum_{d \mid P(z), \, d < D}
    |r_d|
     \ll_{\delta}
     X(\log x)^{-2/3}
\]
provided that $\eta = \eta(\delta)$ is sufficiently small, which we 
assume. 
We fix our sifting level $z$ and sifting variable $s$ at 
\[
 z 
  \defeq D/x^{\delta} = \sqrt{x}/x^{2\delta} 
   \quad 
    \text{and}
     \quad 
      s \defeq \log D/\log z 
         = (1 - 2\delta)/(1 - 4\delta) = 1944/1943.
\]
We can infer from \eqref{eq:sievedim} and 
\cite[Theorem 11.12--Theorem 11.13 et seq.]{FI10} that 
\[
  S(\mathscr{A},z)
   \ge 
    XV(z)
    \Br{f(s) + O\big((\log D)^{-1/6}\big)}
    - R(D,z),
\]
and 
\[
 S(\mathscr{A},z)
  \le 
   XV(z)
    \Br{F(s) + O\big((\log D)^{-1/6}\big)}
     + R(D,z).  
\]
As $V(z) \asymp (\log z)^{1/2}$ by \eqref{eq:mert34} and 
$R(D,z) \ll_{\delta} X(\log z)^{2/3}$, the latter can be subsumed 
under the $O$-term in each case, hence 
\begin{equation}
 \label{eq:bndsSAz}
  f(s)
  + O_{\delta}\big((\log x)^{-1/6}\big)
    \le 
     \frac{S(\mathscr{A},z)}
          {XV(z)}
      \le 
       F(s) 
        + O_{\delta}\big((\log x)^{-1/6}\big).
\end{equation}
Since $S(\mathscr{A},\sqrt{x}) \le S(\mathscr{A},z)$, the upper 
bound in \eqref{eq:sieveup} follows.
%
%
We claim that  
\begin{equation}
 \label{eq:claimlow}
  \frac{S(\mathscr{A},z) 
       - S(\mathscr{A},\sqrt{x})}
       {XV(z)}
    \le 
     {\textstyle \frac{1}{2} }
      f(s)
     + 
      O_{\delta}\big((\log x)^{-1/6}\big),
\end{equation}
which, when combined with the first inequality in 
\eqref{eq:bndsSAz}, gives the lower bound in \eqref{eq:sieveup}.

For $\zh < \sqrt{x}$ we have Buchstab's identity 
\cite[(6.4)]{FI10}:
\[
   S(\mathscr{A},z) - S(\mathscr{A},\sqrt{x})
  =
   \sums[\zh < p_1 \le \sqrt{x}][p_1 \equiv 3 \pod{4}]
    \sums[m \le x][p_1 \mid \, m][(m,P(p_1)) = 1]
     \ind{\PP \cap \, j \pod{8k}}(2km + 1)
      \eqdef 
       T.
\]
Suppose $x^{1/3} \le \zh < \sqrt{x}$ and consider any $m$ that 
makes a nonzero contribution to the inner sum in $T$.
We have $p_1 \mid m$ and $m \le p_1^3$ for some 
$p_1 \equiv 3 \pod{4}$, $m$ is not divisible by any prime less 
than $p_1$ in $\PP \cap 3 \pod{4}$, yet recall that 
$m \equiv 1 \pod{4}$ (for $2km + 1 \equiv j \equiv 3 \pod{8}$ and 
$k \equiv 1 \pod{4}$).
Therefore, $p_2 \mid m$ for exactly one prime 
$p_2 \equiv 3 \pod{4}$ in addition to $p_1$.
Since $(k,p_1p_2) = 1$, we conclude that 
$
 m = ap_1p_2 
$
for some $a,p_1,p_2$ such that  
\[
 a \equiv 1 \pod{4},
  \quad 
   p_1 \equiv p_2 \equiv 3 \pod{4},
    \quad 
     \zh < p_1 \le \sqrt{x}
      \quad 
       \text{and}
        \quad 
     p_1 \le p_2 \le x/(ap_1). 
\]
Also, we have $az^2 < ap_1^2 \le ap_1p_2 \le x$;
in particular,
\[
 a < x/\zh^2 \le \zh < p_1, 
  \quad 
   \ind{\BB}(a) = 1 
  \quad 
   \text{and}
    \quad 
     \zh < p_1 \le \sqrt{x/a}.
\]
Hence 
\[
 T
  \le 
   \sum_{a \le x/\zh^2}
    \ind{\BB}(a)
     \sums[\zh < p_1 \le \sqrt{x/a}][p_1 \equiv 3 \pod{4}] 
      \hspace{5pt}
      \sums[n_2 \le x/(ap_1)]
           [n_2 \equiv 3 \pod{4}]
       \ind{\PP}(n_2)
        \cdot 
         \ind{\PP}
        (2k ap_1n_2 + 1).
\]

We let $(\lambda_{d_2})$ and $(\lambda_d)$ be any upper-bound 
sieves of level $\Dh$ and ``of beta type'' (so that 
$\lambda_{d_2}, \lambda_{d} \in \{-1,0,1\}$).
We note that as 
$
 \ind{\PP}(n)
  \le 
   \sum_{\nu \mid n}
    \lambda_{\nu}
$
($\nu = d_2,d$) for every $n$, we have
\[
 \sums[n_2 \le x/(ap_1)]
      [n_2 \equiv 3 \pod{4}]
  \ind{\PP}(n_2)
   \cdot
    \ind{\PP}
   (2ap_1n_2 + 1)
    \le 
     \sum_{d_2,d}
      \lambda_{d_2}\lambda_{d}
       \sums[n_2 \le x/(ap_1)]
            [n_2 \equiv 3 \pod{4},\, n_2 \equiv 0 \pod{d}]            
            [2ap_1n_2 + 1 \equiv 0 \pod{d_2}]
         1.
\]
The three congruences in the last sum hold only if 
$
(d_2,d) = (d_2,2a) = (2,d) = 1,
$
so combining what we have so far, we obtain (for some residue 
$b \pod{4d_2d}$), 
\begin{align*}
  T
 & \le 
    \sums[a \le x/\zh^2]
     \ind{\BB}(a)
      \sums[\zh < p_1 \le \sqrt{x/a}][p_1 \equiv 3 \pod{4}]
       \sums
            [(d_2,d) = 1]
            [(d_2,2a) = 1]
            [(2,d) = 1]
        \lambda_{d_2}\lambda_{d}
    \sums[n_2 \le x/(ap_1)]
         [n_2 \equiv b \pod{4d_2d}]
     1
 \\
 & =
     \sums[a \le x/\zh^2]
      \ind{\BB}(a)
      \sums[\zh < p_1 \le \sqrt{x/a}][p_1 \equiv 3 \pod{4}]
       \sums
            [(d_2,d) = 1]
            [(d_2,2a) = 1]
            [(2,d) = 1]
             \lambda_{d_2}\lambda_{d}
              \Br{\frac{x}{4ap_1 d_2d} + O(1)}.
\end{align*}

The contribution of the $O$-term to the sum is 
$\ll \Dh^2x/\zh \le \Dh^2\zh^2$.
By a general result \cite[Theorem 5.9]{FI10} on the composition of 
linear sieves, 
\[
 \sums
      [(d_2,d) = 1]
      [(d_2,2a) = 1]
      [(2,d) = 1]
       \frac{\lambda_{d_2}\lambda_{d}}{d_2d}
  \le
   \frac{4C + o(1)}{(\log \Dh)^2}
    \frac{2a}{\phi(2a)}
     \br{\frac{2}{\phi(2)}} 
   \le
     \frac{16C + o(1)}{(\log \Dh)^2}
      \frac{k}{\phi(k)}
       \frac{a}{\phi(a)},
\]
where $o(1)$ denotes a quantity tending to zero as $\Dh$ tends to 
infinity and%
\footnote{%
The constant $\prod_{p}\br{1 + (p-1)^{-2}} = 2.826\ldots$ is 
known as Murata's constant.
}
\[
 \textstyle 
 C
  = 
   \prod_{p\,\nmid\,2a}
    \br{1 + (p-1)^{-2}}
     \le 
      \frac{1}{2}
       \prod_p
        \br{1 + (p-1)^{-2}}
         =
          1.413\ldots.
\]
Thus, $16C + o(1) < 24$ if $\Dh$ is sufficiently large, as we now 
assume.
Gathering all of this, then using the fact that 
$
 \sum_{a \le x/\zh^2}
  \ind{\BB}(a)/\phi(a)
   \le 
    W(x/\zh^2)
$
(cf.\ \eqref{eq:def:VW}) and the bound \eqref{eq:mert1p}, we 
obtain that  
\begin{align*}
 T
 & 
  \le 
   \frac{6x}{\phi(k)(\log \Dh)^2}
    \sum_{a \le x/\zh^2}
     \frac{\ind{\BB}(a)}{\phi(a)}
      \sums[\zh < p_1 \le \sqrt{x}][p_1 \equiv 3 \pod{4}]
       \frac{1}{p_1}
        +
         O(\Dh^2\zh^2) 
 \\
 & 
  \le 
    \frac{3xW(x/\zh^2)\log(x/\zh^2)}{2\phi(k)(\log \Dh)^2\log \zh}
     \br{1 + O\br{\frac{1}{\log(x/\zh^2)}}}
   + O(\Dh^2\zh^2).
\end{align*}

We want to exchange the factor $xW(x/\zh^2)/\phi(k)$ for 
$XV(\zh)$, where recall that $X\defeq\pi(2kx + 1)/\phi(8k)$. 
We have 
$
 x/(2\phi(k))
  =
   X(\log x)(1 + O(1/\log x))  
$
by the prime number theorem. 
By \eqref{eq:sievedim} we have 
\[
 W(x/\zh^2)
  =
   W(\zh)
    \bigg(
     \frac{\log(x/\zh^2)}{\log \zh}
    \bigg)^{1/2}
    \bigg(
     1 + O\bigg(\frac{1}{\log(x/\zh^2)}\bigg)
    \bigg),
\]
and by \eqref{eq:WonV} we have, with $A_1$ and $A_3$ being the 
constants defined in \eqref{eq:def:A0},
\[
 W(\zh)
  =
   {\textstyle \frac{1}{2}}
    A_1A_3
     \e^{\gamma}
    V(\zh)
     (\log \zh)
     \bigg(
      1 + O\bigg(\frac{1}{\log \zh}\bigg)
     \bigg).
\]
Gathering once more we obtain 
\[
 T
  \le 
   {\textstyle \frac{3}{2}} A_1A_3\e^{\gamma}
    XV(\zh)
     \frac{(\log x)\big(\log(x/\zh^2)\big)^{3/2}}
          {(\log \Dh)^2(\log \zh)^{1/2}}
     \bigg(
     1 + O\bigg(\frac{1}{\log(x/\zh^2)}\bigg)
     \bigg)
   + 
      O(\Dh^2\zh^2).
\]

We now set $\Dh \defeq \sqrt{\zh}/x^{\delta}$.
We have 
\[
 x/\zh^2 < \zh, 
 \quad 
  \log \zh
   \asymp
    \log \Dh
     \asymp \log x, 
      \quad 
       \log(x/\zh^2)
        \asymp
         \delta\log x,
          \quad 
           \Dh^2\zh^2 = x^{1 - 2\delta}.
\]
It is therefore apparent that $T \ll \delta^{3/2}XV(\zh)$.
To be more precise, 
\[
 \frac{(\log x)\big(\log(x/\zh^2)\big)^{3/2}}
      {(\log \Dh)^2(\log \zh)^{1/2}}
  =
    (4\delta)^{3/2}(1/4 - 2\delta)^{-2}(1/2 - 2\delta)^{1/2}
  < 
   240\delta^{3/2}.
\]
Finally, it is clear that    
$
 A_1A_3
  \le 
   \prod_{p}
    \big(1 + 1/(p(p - 2))\big)
$
(see the definition \eqref{eq:def:A0} of $A_1$ and $A_3$), and it 
is straightforward to verify that this product is less than 
$\prod_{p}(1 - p^{-2})^{-1} = \pi^2/6$. 
Hence 
\[
 T
  \le 
   60\pi^2\e^{\gamma}
    \delta^{3/2} 
     XV(\zh)
      \Br{
       1 + O\big(1/(\delta\log x)\big)
         }.
\]
A calculation shows that 
$60\pi^2\e^{\gamma}\delta^{3/2} = 0.0043\ldots$ (recall that 
$\delta = 1/3890$), and that by \eqref{eq:fs},
$f(s) = 0.0341\ldots$.
Hence \eqref{eq:claimlow}.
\end{proof}


\end{document}